\documentclass{article}
\usepackage{amsmath,amsfonts,amssymb,amsthm,mathtools,enumitem,geometry,esint,color,hyperref,cleveref}

\setlist[enumerate,1]{label=(\arabic*)}

\theoremstyle{plain}
\newtheorem{theo}{Theorem}[section]

\theoremstyle{definition}

\newtheorem*{exa*}{Example}

\newtheorem{lem}[theo]{Lemma}
\newtheorem*{lem*}{Lemma}

\newtheorem{pro}[theo]{Proposition}
\newtheorem*{pro*}{Proposition}

\newtheorem*{cla*}{Claim}

\newtheorem*{lie*}{Lie}

\theoremstyle{remark}
\newtheorem{rem}[theo]{Remark}
\newtheorem*{rem*}{Remark}

\newtheorem*{note*}{Note}

\allowdisplaybreaks

\crefname{cor}{Corollary}{Corollaries}
\crefname{pro}{Proposition}{Propositions}
\crefname{lem}{Lemma}{Lemmas}
\crefname{theo}{Theorem}{Theorems}
\crefname{rem}{Remark}{Remarks}
\crefname{section}{Section}{Sections}
\crefname{equation}{formula}{formulas}
\crefname{enumi}{}{}

\numberwithin{equation}{section}

\newcommand\intd{\mathop{}\!\mathrm{d}}
\newcommand\df\emph
\newcommand\ind[1]{1_{#1}}
\newcommand\tx\text
\newcommand\f\frac

\newcommand\loc{\mathrm{loc}}

\newcommand\Q{\mathcal Q}
\newcommand\B{\mathcal B}

\newcommand\D{\mathcal D}

\newcommand\Bu{\widetilde{\mathcal B}}
\newcommand\Bc{\mathcal B^{\mathrm c}}

\newcommand\avint\fint

\DeclareMathOperator\var{var}

\DeclareMathOperator\sle{l}
\DeclareMathOperator\prt{prt}

\let\div\relax
\DeclareMathOperator\div{div}

\newcommand\M{\mathrm M}
\newcommand\Mu{\widetilde{\mathrm M}}
\newcommand\Md{\mathrm M^{\mathrm d}}
\newcommand\Mc{\mathrm M^{\mathrm c}}

\newcommand\BV{\mathrm{BV}}
\newcommand\ub\underbrace

\newcommand\la\langle
\newcommand\ra\rangle

\newcommand\comp\complement
\newcommand\mb[1]{\partial_*{#1}}

\newcommand\cl[1]{\overline{#1}}
\newcommand\sm[1]{\mathcal{H}^{d-1}(#1)}

\newcommand\lm[1]{\mathcal{L}(#1)}
\newcommand\lmb[1]{\mathcal{L}\Bigl(#1\Bigr)}

\begin{document}

\title{Endpoint Sobolev Bounds for Fractional Hardy-Littlewood Maximal Operators}
\author{Julian Weigt\footnote{
Aalto University,
Department of Mathematics and Systems Analysis,
P.O.\ Box 11100,
FI-00076 Aalto University,
Finland,
\texttt{julian.weigt@aalto.fi}
}
}
\maketitle

\begin{abstract}
Let \(0<\alpha<d\) and \(1\leq p<d/\alpha\).
We present a proof that for all \(f\in W^{1,p}(\mathbb{R}^d)\)
both the centered and the uncentered Hardy-Littlewood fractional maximal operator
\(\M_\alpha f\) are weakly differentiable and
\(
\|\nabla\M_\alpha f\|_{p^*}
\leq C_{d,\alpha,p}
\|\nabla f\|_p
,
\)
where
\(
p^*
=
(p^{-1}-\alpha/d)^{-1}
.
\)
In particular it covers the endpoint case \(p=1\) for \(0<\alpha<1\)
where the bound was previously unknown.
For \(p=1\) we can replace \(W^{1,1}(\mathbb{R}^d)\) by \(\BV(\mathbb{R}^d)\).
The ingredients used are
a pointwise estimate for the gradient of the fractional maximal function,
the layer cake formula,
a Vitali type argument,
a reduction from balls to dyadic cubes,
the coarea formula,
a relative isoperimetric inequality
and an earlier established result for \(\alpha=0\) in the dyadic setting.
We use that for \(\alpha>0\) the fractional maximal function does not use certain small balls.
For \(\alpha=0\) the proof collapses.

\end{abstract}

\begingroup
\renewcommand\thefootnote{}\footnotetext{%
}%
\addtocounter{footnote}{-1}%
\endgroup

\section{Introduction}
For \(f\in L^1_\loc(\mathbb{R}^d)\) and a ball or cube \(B\), we denote
\[
f_B
=
\f1{\lm B}\int_B|f|
.
\]
The centered Hardy-Littlewood maximal function is defined by
\[
\Mc f(x) = \sup_{r>0}f_{B(x,r)},
\]
and the uncentered Hardy-Littlewood maximal function is defined by
\[
\Mu f(x) = \sup_{B\ni x}f_B
\]
where the supremum is taken over all balls that contain \(x\).
The regularity of a maximal operator was first studied by Kinnunen in 1997. 
He proved in \cite{MR1469106}
that for each \(p>1\) and \(f\in W^{1,p}(\mathbb{R}^d)\) the bound
\begin{equation}
\label{eq_pg1orig}
\|\nabla\M f\|_p
\leq C_{d,p}
\|\nabla f\|_p
\end{equation}
holds for \(\M=\Mc\).
\Cref{eq_pg1orig} also holds for \(\M=\Mu\).
This implies that both Hardy-Littlewood maximal operators are bounded on Sobolev spaces with \(p>1\).
His proof does not apply for \(p=1\).
Note that unless \(f=0\) also
\(\|\M f\|_1\leq C_{d,1}\|f\|_1\)
fails since \(\M f\) is not in \(L^1(\mathbb{R}^d)\).
In \cite{MR2041705} Haj\l{}asz and Onninen asked whether \cref{eq_pg1orig} also holds for \(p=1\) for the centered Hardy-Littlewood maximal operator.
This question has become a well known problem for various maximal operators 
and there has been lots of research on this topic.
So far it has mostly remained unanswered, but there has been some progress.
For the uncentered maximal function and \(d=1\) it has been proved in \cite{MR1898539} by Tanaka 
and later in \cite{MR3310075} by Kurka for the centered Hardy-Littlewood maximal function. 
The proof for the centered maximal function turned out to be much more complicated. 
Aldaz and P\'erez L\'azaro obtained in \cite{MR2276629} the sharp improvement
\(\|\nabla\Mu f\|_{L^1(\mathbb{R})}\leq\|\nabla f\|_{L^1(\mathbb{R})}\)
of Tanaka's result.
For the uncentered Hardy-Littlewood maximal function
Haj\l{}asz's and Onninen's question already also has a positive answer for all dimensions \(d\) in several special cases.
For radial functions Luiro proved it in \cite{MR3800463},
for block decreasing functions Aldaz and P\'erez L\'azaro proved it in \cite{MR2539555} 
and for characteristic functions the author proved it in \cite{weigt2020variationcharacteristic}.
As a first step towards weak differentiability, Haj\l{}asz and Mal\'y proved in \cite{MR2550181}
that for \(f\in L^1(\mathbb{R}^d)\) the centered Hardy-Littlewood maximal function is approximately differentiable.
In \cite{MR2868961} Aldaz, Colzani and P\'erez L\'azaro proved bounds on the modulus of continuity for all dimensions.

A related question is whether the maximal operator is a continuous operator.
Luiro proved in \cite{MR2280193} that for \(p>1\) the uncentered maximal operator is continuous on \(W^{1,p}(\mathbb{R}^d)\).
There is ongoing research for the endpoint case \(p=1\).
For example Carneiro, Madrid and Pierce proved in \cite{MR3695894} that
\(f\mapsto\nabla\Mu f\) is continuous \(W^{1,1}(\mathbb{R})\rightarrow L^1(\mathbb{R})\)
and in \cite{gonzlezriquelme2020bv} Gonz\'alez-Riquelme and Kosz recently improved this to continuity on \(\BV\).
Carneiro, Gonz\'alez-Riquelme and Madrid proved in \cite{carneiro2020sunrise}
that for radial functions \(f\), the operator \(f\mapsto\nabla\Mu f\) is continuous as a map \(W^{1,1}(\mathbb{R}^d)\rightarrow L^1(\mathbb{R}^d)\).

The regularity of maximal operators has also been studied
for other maximal operators
and
on other spaces.
We focus on the endpoint \(p=1\).
In \cite{MR3063097} Carneiro and Svaiter and in \cite{carneiro2019gradient} Carneiro and Gonz\'alez-Riquelme investigated
maximal convolution operators \(\M\) associated to certain partial differential equations.
Analogous to the Hardy-Littlewood maximal operator they proved
\(\|\nabla\M f\|_{L^1(\mathbb{R}^d)}\leq C_d\|\nabla f\|_{L^1(\mathbb{R}^d)}\)
for \(d=1\), and for \(d>1\) if \(f\) is radial.
In \cite{MR3091605} Carneiro and Hughes proved
\(\|\nabla\M f\|_{l^1(\mathbb{Z}^d)}\leq C_d\|f\|_{l^1(\mathbb{Z}^d)}\)
for centered and uncentered discrete maximal operators.
This bound does not hold on \(\mathbb{R}^d\),
but because in the discrete setting we have
\(\|\nabla f\|_{l^1(\mathbb{Z}^d)}\leq C_d\|f\|_{l^1(\mathbb{Z}^d)}\),
it is weaker than the still open
\(\|\nabla\M f\|_{l^1(\mathbb{Z}^d)}\leq C_d\|\nabla f\|_{l^1(\mathbb{Z}^d)}\).
In \cite{MR2328816} Kinnunen and Tuominen proved the boundedness of a discrete maximal operator
in the metric Haj\l{}asz Sobolev space \(M^{1,1}\).
In \cite{MR3891939} Pérez, Picon, Saari and Sousa 
proved the boundedness of certain convolution maximal operators 
on Hardy-Sobolev spaces \(\dot H^{1,p}\) for a sharp range of exponents, including \(p=1\).
In \cite{weigt2020variationdyadic} the author proved 
\(
\var\Md f
\leq C_d
\var f
\)
for the dyadic maximal operator for all dimensions \(d\).

For \(0\leq\alpha\leq d\) the centered fractional Hardy-Littlewood maximal function is defined by
\[
\Mc_\alpha f(x) = \sup_{r>0}r^\alpha f_{B(x,r)}
.
\]
For a ball \(B\) we denote the radius of \(B\) by \(r(B)\).
The uncentered fractional Hardy-Littlewood maximal function is defined by
\[
\Mu_\alpha f(x) = \sup_{B\ni x}r(B)^\alpha f_B
\]
where the supremum is taken over all balls that contain \(x\).
Note that \(\M_\alpha\) does not make much sense for \(\alpha>d\).
For \(\alpha=0\) it is the Hardy-Littlewood maximal function.
The following is the fractional version of \cref{eq_pg1orig}.
\begin{theo}
\label{theo_goal}
Let \(1\leq p<\infty\) and \(0<\alpha<d/p\) and \(\M_\alpha\in\{\Mc_\alpha,\Mu_\alpha\}\).
Then for all \(f\in W^{1,p}(\mathbb{R}^d)\)
we have that \(\M_\alpha f\) is weakly differentiable with
\begin{equation}
\label{eq_generalproblem}
\|\nabla\M_\alpha f\|_{(p^{-1}-\alpha/d)^{-1}}
\leq C_{d,\alpha,p}
\|\nabla f\|_p
\end{equation}
where the constant \(C_{d,\alpha,p}\) depends only on \(d\), \(\alpha\) and \(p\).
In the endpoint \(p=1\) we can replace \(f\in W^{1,1}(\mathbb{R}^d)\) by \(f\in\BV(\mathbb{R}^d)\).
The endpoint result for \(p=d/\alpha\) holds true as well.
\end{theo}

We prove \cref{theo_goal} in \cref{subsec_setuphl}.
The study of the regularity of the fractional maximal operator was initiated by Kinnunen and Saksman.
They proved in \cite[Theorem~2.1]{MR1979008} that \cref{eq_generalproblem} holds for \(0\leq\alpha<d/p\) and \(1<p<\infty\).
They showed
\(
|\nabla\Mc_\alpha f(x)|
\leq
\M_\alpha|\nabla f|(x)
\)
for almost every \(x\in\mathbb{R}^d\),
and then concluded \cref{eq_generalproblem} from the \(L^{(p^{-1}-\alpha/d)^{-1}}\)-boundedness of \(\M_\alpha\),
which fails for \(p=1\).
Another result by Kinnunen and Saksman in \cite{MR1979008}
is that for all \(\alpha\geq1\) we have
\(
|\nabla\Mc_\alpha f(x)|
\leq
(d-\alpha)
\M_{\alpha-1}f(x)
\)
for almost every \(x\in\mathbb{R}^d\).
In \cite{MR3624402} Carneiro and Madrid used this,
the \(L^{d/(d-\alpha)}\)-boundedness of \(\M_{\alpha-1}\),
and Sobolev embedding
to concluded \cref{eq_generalproblem}.
All of this also works for the uncentered fractional maximal function \(\Mu_\alpha\).
The strategy fails for \(\alpha<1\).

Our main result is the extension of \cref{eq_generalproblem} to the endpoint \(p=1\) for \(0<\alpha<1\) which has been an open problem.
Our proof of \cref{theo_goal} also works for \(1\leq\alpha\leq d\),
and further extends to
\(1\leq p<\infty\), \(0<\alpha\leq d/p\).
We present the proof for this range of parameters here,
since it also smoothens out the blowup of the constants for \(p\rightarrow1\) which occurs in the previous proof for \(p>1\).
Note that interpolation is not immediately available for results on the gradient level.
Our approach fails for \(\alpha=0\).
The corner point \(\alpha=0,\ p=1\) is the earlier mentioned question by Haj\l{}asz and Onninen and remains open.
Similarly to Carneiro and Madrid, we begin the proof with a pointwise estimate
\(
|\nabla\M_\alpha f(x)|
\leq
(d-\alpha)
\M_{\alpha,-1}f(x)
\)
which holds for all \(0<\alpha<d\)
for bounded functions.
We estimate \(\M_{\alpha,-1}f\) in \cref{theo_mfdr} and from that conclude \cref{theo_goal}.

For the centered fractional maximal function define
\[
\Bc_\alpha(x)
=
\{B(x,r)\}
\]
where \(r\) is the largest radius such that
\(
\Mc_\alpha f(x)=r^\alpha f_{B(x,r)}
\)
and for the uncentered fractional maximal function define
\[
\Bu_\alpha(x)
=
\bigl\{B:
x\in\cl B
,\ 
r(B)^\alpha f_B
=
\Mu_\alpha f(x)
,\ 
\forall A\supsetneq B
\ 
r(A)^\alpha f_A < \Mu_\alpha f(x) 
\bigr\}
.
\]
Then for almost every \(x\in\mathbb{R}^d\) the sets \(\Bc_\alpha(x)\) and \(\Bu_\alpha(x)\) are nonempty,
i.e.\ the supremum in the definition of the maximal function is attained
in a largest ball \(B\) with \(x\in\cl B\), see \cref{lem_optimalball}.
For \(\B_\alpha\in\{\Bc_\alpha,\Bu_\alpha\}\) denote
\(\B_\alpha=\bigcup_{x\in\mathbb{R}^d}\B_\alpha(x)\).
For \(\beta\in\mathbb{R}\) with \(-1\leq\alpha+\beta<d\)
this allows us to define the following maximal functions
\begin{align*}
\Mc_{\alpha,\beta} f(x) &= \sup_{B\in\Bc_\alpha:x\in\cl B}r(B)^{\alpha+\beta} f_B,\\
\Mu_{\alpha,\beta} f(x) &= \sup_{B\in\Bu_\alpha:x\in\cl B}r(B)^{\alpha+\beta} f_B
\end{align*}
for almost every \(x\in\mathbb{R}^d\).
Note that also for the centered version the supremum is all balls \(B\in\Bc_\alpha\) whose closure contains \(x\),
not only over those centered in \(x\).
\begin{theo}
\label{theo_mfdr}
Let \(1\leq p<\infty\) and \(0<\alpha<d\) and \(\beta\in\mathbb{R}\) with \(0\leq\alpha+\beta+1<d/p\)
and \(\M_{\alpha,\beta}\in\{\Mc_{\alpha,\beta},\Mu_{\alpha,\beta}\}\).
Then for all \(f\in W^{1,p}(\mathbb{R}^d)\)
we have
\[
\|\M_{\alpha,\beta} f\|_{(p^{-1}-(1+\alpha+\beta)/d)^{-1}}
\leq C_{d,\alpha,\beta,p}
\|\nabla f\|_p
\]
where the constant \(C_{d,\alpha,\beta,p}\) depends only on \(d\), \(\alpha\), \(\beta\) and \(p\).
In the endpoint \(p=1\) we can replace \(f\in W^{1,1}(\mathbb{R}^d)\) by \(f\in\BV(\mathbb{R}^d)\).
The endpoint result for \(p=d/(1+\alpha+\beta)\) holds true as well.
\end{theo}

We prove \cref{theo_mfdr} in \cref{sec_uncentered}.
There had also been progress on \(0<\alpha\leq 1\) similarly as for the Hardy-Littlewood maximal operator.
For the uncentered fractional maximal function
Carneiro and Madrid proved \cref{theo_goal} for \(d=1\) in \cite{MR3624402},
and Luiro proved \cref{theo_goal} for radial functions in \cite{MR4001028}.
Beltran and Madrid transfered Luiros result to the centered fractional maximal function in \cite{beltran2019regularity}.
In \cite{MR3912794} Beltran, Ramos and Saari proved \cref{theo_goal} for \(d\geq2\) and a centered maximal operator that only uses balls with lacunary radius and for maximal operators with respect to smooth kernels.
The next step after boundedness is continuity of the gradient of the fractional maximal operator,
as it implies boundedness, but doesn't follow from it.
In \cite{beltran2019endpoint,MR4029798} Beltran and Madrid already proved it for the uncentered fractional maximal operator in the cases where the boundedness is known.

For a dyadic cube \(Q\) we denote by \(\sle(Q)\) the sidelength of \(Q\).
The fractional dyadic maximal function is defined by
\[
\Md_\alpha f(x)
=
\sup_{Q:Q\ni x}
\sle(Q)^\alpha f_Q
,
\]
where the supremum is taken over all dyadic cubes that contain \(x\).
The dyadic maximal operator has enjoyed a bit less attention than its continuous counterparts,
such as the centered and the uncentered Hardy-Littlewood maximal operator.
The dyadic maximal operator is different in the sense that
\cref{eq_generalproblem} only holds for \(\alpha=0\), \(p=1\) and only in the variation sense,
for which \cref{eq_generalproblem} has been proved in \cite{weigt2020variationdyadic}.
But for any other \(\alpha\) and \(p\) \cref{eq_generalproblem} fails
because \(\nabla\Md_\alpha f\) is not a Sobolev function.
We can however prove \cref{theo_mfdrdyadic}, the dyadic analog of \cref{theo_mfdr}.
For \(\alpha\geq0\) and a function \(f\in L^1(\mathbb{R}^d)\) define
\(\Q_\alpha\) to be the set of all cubes \(Q\)
such that for all dyadic cubes \(P\supsetneq Q\) we have
\(\sle(P)^\alpha f_P<\sle(Q)^\alpha f_Q\).
\begin{rem}
In the uncentered setting
one could also define \(\B_\alpha\) in a similar way as \(\Q_\alpha\).
\end{rem}
For \(\beta\in\mathbb{R}\) with \(-1\leq\alpha+\beta<d\) also define in the dyadic setting
\[
\Md_{\alpha,\beta} f(x) = \sup_{Q\in\Q_\alpha:x\in\cl Q}\sle(Q)^{\alpha+\beta}f_Q
.
\]
Then
\begin{theo}
\label{theo_mfdrdyadic}
Let \(1\leq p<\infty\) and \(0<\alpha<d\) and \(\beta\in\mathbb{R}\) with \(0\leq\alpha+\beta+1<d/p\).
Then for all \(f\in W^{1,p}(\mathbb{R}^d)\)
we have
\[
\|\Md_{\alpha,\beta} f\|_{(p^{-1}-(1+\alpha+\beta)/d)^{-1}}
\leq C_{d,\alpha,\beta,p}
\|\nabla f\|_p
\]
where the constant \(C_{d,\alpha,\beta,p}\) depends only on \(d\), \(\alpha\), \(\beta\) and \(p\).
In the endpoint \(p=1\) we can replace \(f\in W^{1,1}(\mathbb{R}^d)\) by \(f\in\BV(\mathbb{R}^d)\).
The endpoint result for \(p=d/(1+\alpha+\beta)\) holds true as well.
\end{theo}
Our main result in the dyadic setting is the following.
\begin{theo}
\label{eq_finitedyadiclinear}
Let \(1\leq p<\infty\) and \(0<\alpha<d\).
Then for all \(f\in W^{1,p}(\mathbb{R}^d)\)
we have
\[
\Biggl(
\sum_{Q\in\Q_\alpha}
(\sle(Q)^{\f dp-1}f_Q)^p
\Biggr)^{\f1p}
\leq C_{d,\alpha,p}
\|\nabla f\|_p
\]
where the constant \(C_{d,\alpha,p}\) depends only on \(d\), \(\alpha\) and \(p\).
In the endpoint \(p=1\) we can replace \(f\in W^{1,1}(\mathbb{R}^d)\) by \(f\in\BV(\mathbb{R}^d)\).
The endpoint result for \(p=\infty\) holds true as well.
\end{theo}
\begin{rem}
Note that in \cref{eq_finitedyadiclinear} we restrict \(0<\alpha<d\) and not \(0<\alpha<d/p\).
\end{rem}
In \cref{subsec_setupdyadic} we conclude \cref{theo_mfdrdyadic} from \cref{eq_finitedyadiclinear},
and in \cref{sec_dyadic} we prove \cref{eq_finitedyadiclinear}.
\begin{rem}
\Cref{eq_finitedyadiclinear} fails for \(\alpha=0\).
However for \(\alpha=0\) and \(p=1\),
a version with \(f_Q\) by replaced by \(f_Q-\lambda_Q\) holds for certain \(\lambda_Q\),
see \cite[Proposition~2.5]{weigt2020variationdyadic}.
\end{rem}

\begin{rem}
For centered, uncentered maximal operator and dyadic maximal operator,
\cref{theo_mfdr,theo_mfdrdyadic,eq_finitedyadiclinear} admit localized versions of the following form.
For \(D\subset\mathbb{R}^d\) we set
\(
\B_\alpha(D)
=
\bigcup_{x\in D}\B_\alpha(x)
\)
and
\(
E
=
\bigcup\{c B:B\in\B_\alpha(D)\}
\)
with some large \(c>1\).
Then \cref{theo_mfdr} also holds in the form
\[
\|\nabla\M_{\alpha,-1} f\|_{L^{(p^{-1}-\alpha/d)^{-1}}(D)}
\leq C_{d,\alpha,p}
\|\nabla f\|_{L^p(E)}
.
\]
\Cref{theo_mfdrdyadic} holds with the dyadic version of \(E\)
and \cref{eq_finitedyadiclinear} where the sum on the left hand side is over any subset \(\Q\subset\Q_\alpha\)
and the integral on the right is over \(\bigcup\{cQ:Q\in\Q\}\).
These localized results directly follow from the same proof as the global results,
if one keeps track of the balls and cubes which are being dealt with.
The respective localized version of \cref{theo_goal} can be proven
if one has \cref{lem_gradientbound} without the differentiability assumption.
Then in the reduction of \cref{theo_goal} to \cref{theo_mfdr}
one could apply \cref{theo_mfdr} to the same function \(f\) and \(\Q_\alpha\) for which one is showing \cref{theo_goal},
bypassing the approximation step and
therefore preserving the locality of \cref{theo_mfdr}.
This is in contrast to the actual local fractional maximal operator,
for whom \cref{theo_goal} fails by \cite[Example~4.2]{MR3319617}, which works for \(\alpha>0\).
However if \(\alpha=0\) and \(p>1\) then the local fractional maximal operator is again bounded due to \cite{MR1650343},
and by \cite{weigt2020variationcharacteristic} for \(\alpha=0\) and \(p=1\) and characteristic functions.
\end{rem}

Dyadic cubes are much easier to deal with than balls,
but the dyadic version still serves as a model case for the continuous versions since both versions share many properties.
This can be observed in \cite{weigt2020variationcharacteristic},
where we proved \(\var\M_0\ind E\leq C_d\var\ind E\) for the dyadic maximal operator and the uncentered Hardy-Littlewood maximal operator.
The proof for the dyadic maximal operator is much shorter,
but the same proof idea also works for the uncentered maximal operator.
Also in this paper a part of the proof of \cref{theo_mfdrdyadic} for the dyadic maximal operator
is used also in the proof of \cref{theo_mfdr} for the Hardy-Littlewood maximal operator.

The plan for the proof of \cref{theo_goal}
is the following.
For simplicity we write it down for \(p=1\).
\begin{align*}
\int|\nabla\M_\alpha f|^{\f d{d-\alpha}}
&\leq(d-\alpha)^{\f d{d-\alpha}}
\int(\M_{\alpha,-1}f)^{\f d{d-\alpha}}
\\
&=
d(d-\alpha)^{\f \alpha{d-\alpha}}
\int_0^\infty\lambda^{\f \alpha{d-\alpha}}\lm{\{\M_{\alpha,-1}f>\lambda\}}\intd\lambda
\\
&=
d(d-\alpha)^{\f \alpha{d-\alpha}}
\int_0^\infty\lambda^{\f \alpha{d-\alpha}}\lm{\bigcup\{\cl B:B\in\B_\alpha,r(B)^{\alpha-1}f_B>\lambda\}}\intd\lambda
\\
&\lesssim_\alpha
\int_0^\infty\lambda^{\f \alpha{d-\alpha}}\sum_{B\in\tilde\B_\alpha,cr(B)^{\alpha-1}f_B>\lambda}\lm B\intd\lambda
\\
&=
\sum_{B\in\tilde\B_\alpha}\int_0^{cr(B)^{\alpha-1}f_B}\lambda^{\f\alpha{d-\alpha}}\intd\lambda
\\
&=
\f{(1-\alpha/d)c^{\f d{d-\alpha}}}{(d\sigma_d)^{\f d{d-\alpha}}}
\sum_{B\in\tilde\B_\alpha}(f_B\sm{\partial B})^{\f d{d-\alpha}}
\\
&\leq
\f{(1-\alpha/d)c^{\f d{d-\alpha}}}{(d\sigma_d)^{\f d{d-\alpha}}}
\biggl(
\sum_{B\in\tilde\B_\alpha}f_B\sm{\partial B}
\biggr)^{\f d{d-\alpha}}
\\
&\lesssim_\alpha
\biggl(
\sum_{Q\in\tilde\Q_\alpha}f_Q\sm{\partial Q}
\biggr)^{\f d{d-\alpha}}
\\
&\leq C_{d,\alpha}
(\var f)^{\f d{d-\alpha}}
,
\end{align*}
where \(\sigma_d\) is the volume of the \(d\)-dimensional unit ball.
In the second step we apply the layer cake formula,
in the forth step we pass from a union of arbitrary balls to very disjoint balls \(\tilde\B_\alpha\) with a Vitali covering argument,
in the eighth step
we pass from those balls to comparable dyadic cubes
and as the last step use a result from the dyadic setting.

We use \(\alpha>0\) as follows.
Let \(A\) be a ball and \(B(x,r)\) be a smaller ball that intersects \(A\).
Then by \(A\subset B(x,3r(A))\) we have \(3^{\alpha-d} r(A)^\alpha f_A\leq (3r(A))^\alpha f_{B(x,3r(A))}\).
Thus if \(r^\alpha f_{B(x,r)}\leq3^{\alpha-d} r(A)^\alpha f_A\)
then \(B(x,r)\) is not used by the fractional maximal operator.
Hence it suffices to consider balls \(B\) with \(3^{d-\alpha}(r(B)/r(A))^\alpha f_B > f_A\).
From that we can conclude \(f_B>2f_A\) or \(r(B)\gtrsim_\alpha r(A)\).
Thus for any two balls \(B,A\) used by the fractional maximal operator, one of the following alternatives applies.
\begin{enumerate}
\item
The balls \(B\) and \(A\) are disjoint.
\item 
The intervals \((f_B/2,f_B)\) and \((f_A/2,f_A)\) are disjoint.
\item
The radii \(r(B)\) and \(r(A)\) are comparable.
\end{enumerate}
We use this in the forth step of the proof strategy above.
We use a dyadic version of these alternatives in last step.
Note that for \(\alpha=0\) optimal balls \(B\) of arbitrarily different sizes with similar values \(f_B\) can intersect.

\begin{rem}
There is a proof of \cref{theo_goal}
which has a structure parallel to the one presented above,
but three steps are replaced.
The estimate
\(|\nabla\M_\alpha f|^{\f d{d-\alpha}}\leq(d-\alpha)^{\f d{d-\alpha}}\M_{\alpha,-1} f\)
is replaced by
\(|\nabla\M_\alpha f|^{\f d{d-\alpha}}\leq(d-\alpha)^{\f \alpha{d-\alpha}}|\nabla\M_\alpha f|(\M_{\alpha,-1} f)^{\f\alpha{d-\alpha}}\),
the layer cake formula is replaced by the coarea formula \cite[Theorem~3.11]{MR3409135}
and the Vitali covering argument is replaced by \cite[Lemma~4.1]{weigt2020variationcharacteristic}
which deals with the boundary of balls instead of their volume.
Otherwise it is identical to the proof presented in this paper.
\begin{align*}
\int|\nabla\M_\alpha f|^{\f d{d-\alpha}}
&\leq(d-\alpha)^{\f\alpha{d-\alpha}}
\int|\nabla\M_\alpha f|(\M_{\alpha,-1}f)^{\f\alpha{d-\alpha}}
\\
&=(d-\alpha)^{\f\alpha{d-\alpha}}
\int_0^\infty\int_{\mb{\{\M_\alpha f>\lambda\}}}(\M_{\alpha,-1}f)^{\f\alpha{d-\alpha}}\intd\lambda
\\
&=(d-\alpha)^{\f\alpha{d-\alpha}}
\int_0^\infty\int_{\mb{\bigcup\{\cl B:B\in\B_\alpha,r(B)^\alpha f_B>\lambda\}}}(r(B_x)^{\alpha-1}f_{B_x})^{\f\alpha{d-\alpha}}\intd\sm x\intd\lambda
\\
&\lesssim_\alpha
\int_0^\infty\sum_{B\in\tilde\B_\alpha,r(B)^\alpha f_B>\lambda}\sm{\partial B}(r(B)^{\alpha-1}f_B)^{\f\alpha{d-\alpha}}\intd\lambda
\\
&\lesssim_\alpha
\sum_{B\in\tilde\B_\alpha}(f_B\sm{\partial B})^{\f d{d-\alpha}}
\end{align*}
and from there on arrive exactly as before at the bound by \((\var f)^{\f d{d-\alpha}}\).
This motivates a similar replacement in the dyadic setting.
Instead of proving the boundedness of \(\|\M_{\alpha,-1}f\|_{d/(d-\alpha)}\), \cref{theo_mfdrdyadic},
one might bound 
\[
\int_0^\infty\int_{\mb{\{\M_\alpha f>\lambda\}}}(\M_{\alpha,-1}f)^{\f\alpha{d-\alpha}}\intd\lambda
.
\]
Note that formally 
\[
\int
|\nabla\M_\alpha f(x)|(\M_{\alpha,-1} f(x))^{\f\alpha{d-\alpha}}
\intd x
\]
is not well defined because \(\M_{\alpha,-1}f\) jumps where \(\nabla\M_\alpha f\) is supported.
\end{rem}

\begin{rem}
In the proof of \cref{theo_goal,theo_mfdr,eq_finitedyadiclinear,theo_mfdrdyadic}
we do not a priori need \(f\in L^p(\mathbb{R}^d)\),
it suffices to have \(f\in L^q(\mathbb{R}^d)\) for some \(1\leq q\leq p\).
However from \(\|\nabla f\|_p<\infty\)
we can then anyways conclude \(f\in L^p(\mathbb{R}^d)\) by Sobolev embedding.
\end{rem}

\paragraph{Acknowledgements}
I would like to thank my supervisor, Juha Kinnunen, for all of his support.
I would like to thank Olli Saari for introducing me to this problem.
I am also thankful for the discussions with Juha Kinnunen, Panu Lahti and Olli Saari
who made me aware of a version of the coarea formula \cite[Theorem~3.11]{MR3409135},
which was used in the first draft of the proof,
and for discussions with David Beltran, Cristian Gonz\'alez-Riquelme and Jose Madrid, in particular about the centered fractional maximal operator.
The author has been supported by the Vilho, Yrj\"o and Kalle V\"ais\"al\"a Foundation of the Finnish Academy of Science and Letters.

\section{Reformulation}
\label{sec_setup}
In order to avoid writing absolute values,
we consider only nonnegative functions \(f\) for the rest of the paper.
We can still conclude \cref{theo_goal,theo_mfdr,theo_mfdrdyadic,eq_finitedyadiclinear} for signed functions because
\(|f|_B=f_B\)
and
\(\bigl|\nabla|f|(x)\bigr|\leq|\nabla f(x)|\).
Recall the set of dyadic cubes
\[
\bigcup_{n\in\mathbb{Z}}
\Bigl\{
[x_1,x_1+2^n)\times\ldots\times[x_d,x_d+2^n):\forall i\in\{1,\ldots,n\}\ x_i\in2^n\mathbb{Z}
\Bigr\}
.
\]
For a set \(\B\) of balls or dyadic cubes we denote
\[
\bigcup\B
=
\bigcup_{B\in\B}B
\]
as is commonly used in set theory.
By \(a\lesssim_{\gamma_1,\ldots,\gamma_n} b\) we mean that there exists a constant \(C_{d,\gamma_1,\ldots,\gamma_n}\)
that depends only on the values of \(\gamma_1,\ldots,\gamma_n\) and the dimension \(d\)
and such that \(a\leq C_{d,\gamma_1,\ldots,\gamma_n} b\).

We work in the setting of functions of bounded variation, as in Evans-Gariepy \cite[Section~5]{MR3409135}.
For an open set \(\Omega\subset\mathbb{R}^d\) a function \(u\in L^1_\loc(\Omega)\) is said to have locally bounded variation
if for each open and compactly supported \(V\subset\Omega\) we have
\[\sup\Bigl\{\int_Vu\div\varphi:\varphi\in C^1_{\tx c}(V;\mathbb{R}^d),\ |\varphi|\leq1\Bigr\}<\infty.\]
Such a function comes with a measure \(\mu\) and a function \(\nu:\Omega\rightarrow\mathbb{R}^d\) that has \(|\nu|=1\) \(\mu\)-a.e.\ such that for all \(\varphi\in C^1_{\tx c}(\Omega;\mathbb{R}^d)\) we have
\[\int u\div\varphi=\int\varphi\nu\intd \mu.\]
We denote \(\nabla u=-\nu\mu\)
and define the variation of \(u\) by
\[
\var_\Omega u
=
\mu(\Omega)
=
\|\nabla u\|_{L^1(\Omega)}
.
\]
If \(\nabla u\) is a locally integrable function we call \(u\) weakly differentiable.

\begin{lem}
\label{lem_convergence}
Let \(1<p\leq\infty\) and \((u_n)_n\) be a sequence of locally integrable functions with
\[
\sup_n
\|\nabla u_n\|_p
<\infty
\]
which converge to \(u\) in \(L^1_\loc(\mathbb{R}^d)\).
Then \(u\) is weakly differentiable and
\[
\|\nabla u\|_p
\leq
\limsup_n
\|\nabla u_n\|_p
.
\]
\end{lem}
\begin{proof}
By the weak compactness of \(L^p(\mathbb{R}^d)\) there is a subsequence,
for simplicity also denoted by \((u_n)_n\),
and a \(v\in L^p(\mathbb{R}^d)^d\)
such that \(\nabla u_n\rightarrow v\) weakly in \(L^p(\mathbb{R}^d)\)
and \(\|v\|_p\leq\limsup_n\|\nabla u_n\|_p\).
Let \(\varphi\in C^\infty_c(\mathbb{R}^d)\) and \(i\in\{1,\ldots,d\}\).
Then
\begin{align*}
\int u\partial_i\varphi
=
\lim_{n\rightarrow\infty}
\int u_n\partial_i\varphi
=
-
\lim_{n\rightarrow\infty}
\int \partial_iu_n\varphi
=
-
\int v_i\varphi
\end{align*}
which means \(\nabla u=v\).
\end{proof}

\subsection{Hardy-Littlewood Maximal Operator}\label{subsec_setuphl}

In this section we reduce \cref{theo_goal} to \cref{theo_mfdr}.
Let \(1\leq p<d/\alpha\) and \(f\in L^p(\mathbb{R}^d)\).
For \(x\in\mathbb{R}^d\)
consider for the uncentered maximal operator
the set of balls \(B\) with \(x\in\cl B\)
and 
\(
\M_\alpha f(x)
=
r(B)^\alpha f_B
,
\)
and for the centered maximal operator such balls \(B\) which are centered in \(x\).
Recall that we denote by \(\B_\alpha(x)\)
the subset of those balls that have the largest radius.

\begin{lem}
\label{lem_optimalball}
Let \(\M_\alpha\in\{\Mc_\alpha,\Mu_\alpha\}\)
and \(1\leq p<d/\alpha\).
Let \(f\in L^p(\mathbb{R}^d)\)
and \(x\in\mathbb{R}^d\) be a Lebesgue point of \(f\).
Then \(\B_\alpha(x)\) is nonempty.
\end{lem}
\begin{proof}
We formulate one proof that works both for the centered and uncentered fractional maximal operator.
Let \((B_n)_n\) a sequence of balls with \(x\in B_n\) and
\[
\M_\alpha f(x)=\lim_{n\rightarrow\infty}r(B_n)^\alpha f_{B_n}
.
\]
Assume there is a subsequence \((n_k)_k\) with \(r(B_{n_k})\rightarrow0\).
Then \(f_{B_{n_k}}\rightarrow f(x)\)
and thus
\[
\limsup_{k\rightarrow\infty} r(B_{n_k})^\alpha f_{B_{n_k}}
\leq
f(x) \limsup_{n\rightarrow\infty} r(B_{n_k})^\alpha
=0
,
\]
a contradiction.
Assume there is a subsequence \((n_k)_k\) with \(r(B_{n_k})\rightarrow\infty\).
Then
\begin{align*}
\limsup_{k\rightarrow\infty} r(B_{n_k})^\alpha f_{B_{n_k}}
&\leq
\limsup_{k\rightarrow\infty}
r(B_{n_k})^\alpha
\lm{B_{n_k}}^{-1}
\lm{B_{n_k}}^{1-\f1p}
\Bigl(\int_{B_{n_k}} f^p\Bigr)^{\f1p}
\\
&=
\limsup_{k\rightarrow\infty}
\sigma_d^{-\f1p}
r(B_{n_k})^{\alpha-\f dp}
\Bigl(\int_{B_{n_k}} f^p\Bigr)^{\f1p}
\\
&\leq
\sigma_d^{-\f1p}
\limsup_{k\rightarrow\infty}
r(B_{n_k})^{\alpha-\f dp}
\|f\|_p
=
0
\end{align*}
since \(\|f\|_p<\infty\),
a contradiction.
Hence there is a subsequence \((n_k)_k\) such that \(r(B_{n_k})\) converges to some value \(r\in(0,\infty)\).
We can conclude that there is a ball \(B\) with \(x\in\cl B\) and \(r(B)=r\) and
\(
\int_{B_{n_k}} f
\rightarrow
\int_B f
.
\)
So we have
\[
\M_\alpha f(x)
=
\lim_{k\rightarrow\infty} r(B_{n_k})^\alpha f_{B_{n_k}}
=
r(B)^\alpha f_B
.
\]
A similar argument shows that there exist a largest ball \(B\) for which \(\sup_{\cl B\ni x}r(B)^\alpha f_B\) is attained.
\end{proof}

\begin{lem}
\label{lem_continuous}
Let \(\M_\alpha\in\{\Mc_\alpha,\Mu_\alpha\}\).
and \(f\in L^\infty(\mathbb{R}^d)\) have bounded variation.
Then \(\M_\alpha f\) is locally Lipschitz.
\end{lem}
\begin{proof}
If \(f=0\) then the statement is obvious,
so consider \(f\neq0\).
Let \(B\) be a ball.
Then there is a ball \(A\supset B\) with \(f_A>0\).
Define
\[
r_0
=
2r(A)
\Bigl(
\f{f_A}{2^d\|f\|_\infty}
\Bigr)^{1/\alpha}
\]
and let \(x\in B\).
Then \(A\subset B(x,2r(A)\)
so that for \(r<r_0\) we have
\[
r^\alpha f_{B(x,r)}
<
(2r(A))^\alpha \f{f_A}{2^d\|f\|_\infty}\|f\|_\infty
\leq
(2r(A))^\alpha f_{B(x,2r(A))}
.
\]
That means that on \(B\) the maximal function \(\M_\alpha f\)
is the supremum over all functions \(\sigma_d^{-1}r^{\alpha-d}f*\ind{B(z,r)}\)
with \(r\geq r_0\) and \(z\) such that \(0\in B(z,r)\) for the uncentered operator and \(z=0\) for the centered.
Those convolutions are weakly differentiable with
\[
\nabla(r^{\alpha-d}f*\ind{B(z,r)})
=
r^{\alpha-d}(\nabla f)*\ind{B(z,r)}
\]
so that
\[
|\nabla(r^{\alpha-d}f*\ind{B(z,r)})|
\leq
r^{\alpha-d}\var f
\leq
r_0^{\alpha-d}\var f
.
\]
Thus on \(B\) the maximal function \(\M_\alpha f\)
is a supremum of functions with Lipschitz constant \(\sigma_d^{-1}r_0^{\alpha-d}\var f\)
and hence itself Lipschitz with the same constant.
\end{proof}

The following has essentially already been observed in \cite{MR3319617,MR1979008,MR2280193,MR4001028}.
\begin{lem}
\label{lem_gradientbound}
Let \(\M_\alpha\in\{\Mc_\alpha,\Mu_\alpha\}\) and
let \(\M_\alpha f\) be differentiable in \(x\).
Then for every \(B\in\B_\alpha(x)\) we have
\[
|\nabla\M_\alpha f(x)|
\leq
(d-\alpha)r(B)^{\alpha-1}f_B
.
\]
In the uncentered case
if \(x\in B\) we have
\(
\nabla\Mu_\alpha f(x)=0
.
\)
\end{lem}
\begin{proof}
Let \(B(z,r)\in\B_\alpha(x)\)
and let \(e\) be a unit vector.
Note that for the centered maximal operator we have \(z=x\).
Then for all \(h>0\) we have \(x+he\in\cl{B(z,r+h)}\).
Thus
\begin{align*}
|\nabla\M_\alpha  f(x)|
&=
\sup_e\lim_{h\rightarrow0}
\f{\M_\alpha f(x)-\M_\alpha f(x+he)}h
\\
&\leq
\f1{\sigma_d}
\lim_{h\rightarrow0}
\f1h(r^{\alpha-d}\int_{B(z,r)}f-(r+h)^{\alpha-d}\int_{B(z+eh,r+h)}f)
\\
&\leq
\f1{\sigma_d}
\lim_{h\rightarrow0}
\f1h(r^{\alpha-d}\int_{B(z,r)}f-(r+h)^{\alpha-d}\int_{B(z,r)}f)
\\
&=
\f1{\sigma_d}
\lim_{h\rightarrow0}
\f1h(r^{\alpha-d}-(r+h)^{\alpha-d})\int_{B(z,r)}f
\\
&=
\f1{\sigma_d}
(d-\alpha)r^{\alpha-d-1}\int_{B(z,r)}f
.
\end{align*}
If \(x\in B(z,r)\) then since for all \(y\in B(z,r)\) we have \(\M_\alpha f(y)\geq\M_\alpha f(x)\)
we get \(\nabla\M_\alpha f(x)=0\).
\end{proof}

Now we reduce \cref{theo_goal} to \cref{theo_mfdr}.
We prove \cref{theo_mfdr} in \cref{sec_uncentered}.

\begin{proof}[Proof of \cref{theo_goal}]
For each \(n\in\mathbb{N}\) define a cutoff function \(\varphi_n\) by
\[
\varphi_n(x)
=
\begin{cases}
1,&0\leq|x|\leq 2^n,\\
2-2^{-n}|x|,&2^n\leq |x|\leq2^{n+1},\\
0,&2^{n+1}\leq|x|<\infty.
\end{cases}
\]
Then
\(|\nabla\varphi_n(x)|=2^{-n}\ind{2^n\leq|x|\leq2^{n+1}}\)
and thus
\begin{equation}
\label{eq_cutoffgradient}
\|f\nabla\varphi_n\|_p
=
2^{-n}\|f\|_{L^p(B(0,2^{n+1})\setminus B(0,2^n))}
\rightarrow0
\end{equation}
for \(n\rightarrow\infty\).
Denote \(f_n(x)=\min\{f(x),n\}\cdot \varphi_n(x)\).
Then by \cref{eq_cutoffgradient} we have
\begin{equation}
\label{eq_functionapproximation}
\lim_{n\rightarrow\infty}
\|\nabla f_n\|_p
=
\lim_{n\rightarrow\infty}
\|\nabla f_n-\min\{f,n\}\nabla\varphi_n\|_p
=
\lim_{n\rightarrow\infty}
\|\varphi_n\nabla\min\{f,n\}\|_p
=
\|\nabla f\|_p
.
\end{equation}
Since \(1\leq p<d/\alpha\) and \(f\in L^p(\mathbb{R}^d)\) we have \(\M_\alpha f\in L^{(p^{-1}-\alpha/d)^{-1},\infty}(\mathbb{R}^d)\subset L^1_\loc(\mathbb{R}^d)\).
Then since \(\M_\alpha f_n\rightarrow\M_\alpha f\) pointwise from below,
\(\M_\alpha f_n\) converges to \(\M_\alpha f\) in \(L^1_\loc(\mathbb{R}^d)\).
So from \cref{lem_convergence} it follows that
\[
\|\nabla\M_\alpha f\|_{(p^{-1}-\alpha/d)^{-1}}
\leq
\limsup_{n\rightarrow\infty}
\|\nabla\M_\alpha f_n\|_{(p^{-1}-\alpha/d)^{-1}}
.
\]
By \cref{lem_continuous} we have that \(\M_\alpha f_n\) is weakly differentiable
and differentiable almost everywhere,
so that by \cref{lem_optimalball,lem_gradientbound,theo_mfdr} we have
\begin{align*}
\int|\nabla\M_\alpha f_n|^{(p^{-1}-\alpha/d)^{-1}}
&\leq
(d-\alpha)
\|\M_\alpha f_n/r(B_x)\|_{(p^{-1}-\alpha/d)^{-1}}
\\
&\leq
(d-\alpha)
\|\M_{\alpha,-1} f_n\|_{(p^{-1}-\alpha/d)^{-1}}
\\
&\lesssim_\alpha
\|\nabla f_n\|_p
,
\end{align*}
which by \cref{eq_functionapproximation} converges to
\(
\|\nabla f\|_p
.
\)
for \(n\rightarrow\infty\).
For the endpoint \(p=d/\alpha\) the proof works the same.
\end{proof}

\subsection{Dyadic Maximal Operator}\label{subsec_setupdyadic}

In this section we reduce \cref{theo_mfdrdyadic} to \cref{eq_finitedyadiclinear}.
Let \(1\leq p<d/\alpha\) and \(f\in L^p(\mathbb{R}^d)\).
Recall that we denote by \(\Q_\alpha\) the set of all dyadic cubes \(Q\)
such that for every dyadic cube ball \(P\supsetneq Q\) we have \(\sle(P)^\alpha f_P<\sle(Q)^\alpha f_Q\).
For \(x\in\mathbb{R}^d\),
we denote by \(\Q_\alpha(x)\)
the set of dyadic cubes \(Q\) with \(x\in\cl Q\)
and 
\[
\Md_\alpha f(x)=\sle(Q)^\alpha f_Q
.
\]

\begin{lem}
\label{lem_optimacube}
Let \(1\leq p<d/\alpha\) and \(f\in L^p(\mathbb{R}^d)\)
and \(x\in\mathbb{R}^d\) be a Lebesgue point of \(f\).
Then \(\Q_\alpha(x)\) contains a dyadic cube \(Q_x\) with
\[
\sle(Q_x)
=
\sup_{Q\in\Q_\alpha(x)}\sle(Q)
\]
and that cube also belongs to \(\Q_\alpha\).
\end{lem}
\begin{proof}
Let \((Q_n)_n\) be a sequence of cubes with \(\sle(Q_n)\rightarrow\infty\).
Then
\begin{align*}
\limsup_{n\rightarrow\infty} \sle(Q_n)^\alpha f_{Q_n}
&\leq
\limsup_{n\rightarrow\infty}
\sle(Q_n)^{\alpha-d}
\lm{Q_n}^{1-\f1p}
\Bigl(\int_{Q_n} f^p\Bigr)^{\f1p}
\\
&=
\limsup_{n\rightarrow\infty}
\sle(Q_n)^{\alpha-d+d-\f dp}
\Bigl(\int_{Q_n} f^p\Bigr)^{\f1p}
\\
&=
\limsup_{n\rightarrow\infty}
\sle(Q_n)^{\alpha-\f dp}
\Bigl(\int_{Q_n} f^p\Bigr)^{\f1p}
\\
&\leq
\limsup_{n\rightarrow\infty}
\sle(Q_n)^{\alpha-\f dp}
\|f\|_p
=
0
.
\end{align*}
Let \((Q_n)_n\) be a sequence of cubes with \(\sle(Q_n)\rightarrow0\).
Then since \(f_{Q_n}\rightarrow f(x)\) and \(\sle(Q_n)^\alpha\rightarrow0\), we have \(\sle(Q_n)^\alpha f_Q\rightarrow0\).
Thus since for each \(k\) there are at most \(2^d\) many cubes \(Q\)
with \(\sle(Q)=2^k\)
and whose closure contains \(x\),
the supremum has to be attained for a finite set of cubes
from which we can select the largest.
\end{proof}

Now we reduce \cref{theo_mfdrdyadic} to \cref{eq_finitedyadiclinear}.
We prove \cref{eq_finitedyadiclinear} in \cref{sec_dyadic}.

\begin{proof}[Proof of \cref{theo_mfdrdyadic}]
By \cref{lem_optimacube}, \(\Md_{\alpha,\beta}f\) is defined almost everywhere.
We have
\begin{align*}
\int
(\Md_{\alpha,\beta} f(x))^{(p^{-1}-(1+\alpha+\beta)/d)^{-1}}
\intd x
&\leq
\int
\sum_{Q\in\Q_\alpha}
\ind Q(x)
(\sle(Q)^{\alpha+\beta}f_Q)^{(p^{-1}-(1+\alpha+\beta)/d)^{-1}}
\intd x
\\
&=
\sum_{Q\in\Q_\alpha}
\lm Q
(\sle(Q)^{\alpha+\beta}f_Q)^{(p^{-1}-(1+\alpha+\beta)/d)^{-1}}
\\
&=
\sum_{Q\in\Q_\alpha}
(\sle(Q)^{d/p-1}f_Q)^{(p^{-1}-(1+\alpha+\beta)/d)^{-1}}
\\
&\leq
\biggl(
\sum_{Q\in\Q_\alpha}
\bigl(
\sle(Q)^{d/p-1}f_Q
\bigr)^p
\biggr)^{(1-p(1+\alpha+\beta)/d)^{-1}}
\\
&\lesssim_\alpha
\|\nabla f\|_p^{(p^{-1}-(1+\alpha+\beta)/d)^{-1}},
\end{align*}
where the last step follows from \cref{eq_finitedyadiclinear}.
In the endpoint case we have
by \cref{eq_finitedyadiclinear}
\begin{align*}
\|\Md_{\alpha,\beta} f\|_\infty
&=
\sup_{Q\in\Q_\alpha}
\sle(Q)^{\alpha+\beta}f_Q
=
\sup_{Q\in\Q_\alpha}
\sle(Q)^{\f dp-1}f_Q
\leq
\Biggl(
\sum_{Q\in\Q_\alpha}
(\sle(Q)^{\f dp-1}f_Q)^p
\Biggr)^{\f1p}
\lesssim_p
\|\nabla f\|_p
.
\end{align*}
\end{proof}

\section{Dyadic Maximal Operator}
\label{sec_dyadic}
In this section we prove \cref{eq_finitedyadiclinear}.
For a measurable set \(E\subset\mathbb{R}^d\) we define the measure theoretic boundary by
\[
\mb E
=
\Bigl\{x:\limsup_{r\rightarrow0}\f{\lm{B(x,r)\setminus E}}{r^d}>0,\ \limsup_{r\rightarrow0}\f{\lm{B(x,r)\cap E}}{r^d}>0\Bigr\}
.
\]
We denote the topological boundary by \(\partial E\).
As in \cite{weigt2020variationdyadic,weigt2020variationcharacteristic},
our approach to the variation is the coarea formula rather then the definition of the variation,
see for example \cite[Theorem~5.9]{MR3409135}.
\begin{lem}
\label{lem_coarearange}
Let \(f\in L^1_\loc(\mathbb{R}^d)\) with locally bounded variation and \(U\subset\mathbb{R}^d\).
Then
\[
\var_U f
=
\int_\mathbb{R}\sm{\mb{\{f>\lambda\}}\cap U}\intd\lambda
.
\]
\end{lem}
\begin{lem}
\label{cor_coarearange}
Let \(f\in L^1_\loc(\mathbb{R}^d)\) be weakly differentiable and \(U\subset\mathbb{R}^d\) and \(\lambda_0<\lambda_1\).
Then
\[
\int_{\{x\in U:\lambda_0<f(x)<\lambda_1\}}
|\nabla f|
=
\int_{\lambda_0}^{\lambda_1}\sm{\mb{\{f>\lambda\}}\cap U}\intd\lambda
.
\]
\end{lem}
Recall also the relative isoperimetric inequality for cubes.
\begin{lem}
\label{lem_isoperimetric}
Let \(Q\) be a cube and \(E\) be a measurable set.
Then
\[
\min\{\lm{Q\cap E},\lm{Q\setminus E}\}^{d-1}\lesssim\sm{\mb E\cap Q}^d
.
\]
\end{lem}

We will use a result from the case \(\alpha=0\).
For a subset \(\Q\subset\Q_0\) and \(Q\in\Q_0\),
we denote
\[
\lambda_Q^\Q
=
\min\biggl\{
\max\Bigl\{
\inf\{\lambda:\lm{\{f>\lambda\}\cap Q}<2^{-d-2}\lm Q\}
,\ 
\sup\{f_P:P\in\Q,\ P\supsetneq Q\}
\Bigr\}
,f_Q
\biggr\}
.
\]
\begin{pro}
\label{pro_densitylow}
Let \(1\leq p<\infty\) and \(f\in L^1_\loc(\mathbb{R}^d)\) and \(|\nabla f|\in L^p(\mathbb{R}^d)\).
Then for every set \(\Q\subset\Q_0\) we have
\[
\sum_{Q\in\Q}
(\sle(Q)^{\f dp-1}(f_Q-\lambda_Q^\Q))^p
\lesssim_p
\|\nabla f\|_p^p
.
\]
For \(p=1\) it also holds with \(\|\nabla f\|_1\) replaced by \(\var f\).
\end{pro}
\begin{rem}
\label{rem_largeralpha}
We have that \(\alpha<\beta\) implies \(\Q_\beta\subset\Q_\alpha\).
This is because for \(\sle(Q)<\sle(P)\),
\(\sle(Q)^\alpha f_Q>\sle(P)^\alpha f_P\)
becomes a stronger estimate the larger \(\alpha\) becomes.
\end{rem}
By \cref{rem_largeralpha} we can apply \cref{pro_densitylow} to \(Q=\Q_\alpha\).
For \(p=1\) \cref{pro_densitylow} is Proposition~2.5 in \cite{weigt2020variationdyadic}.
For the proof for all \(p\geq1\) we follow the strategy in \cite{weigt2020variationdyadic}.
In particular we use the following result.
For \(Q\in\Q_0\) we denote
\[
\bar \lambda_Q
=
\min\biggl\{
\max\Bigl\{
\inf\{\lambda:\lm{\{f>\lambda\}\cap Q}<\lm Q/2\}
,\ 
\sup\{f_P:P\in\Q_0,\ P\supsetneq Q\}
\Bigr\}
,f_Q
\biggr\}
.
\]
\begin{lem}[Corollary~3.3 in \cite{weigt2020variationdyadic}]
\label{cla_mostmasssparseabove}
Let \(f\in L^1_\loc(\mathbb{R}^d)\).
Then for every \(Q\in\Q_0\)
we have
\[
\lm Q(f_Q-\lambda_Q^\emptyset)
\leq2^{d+2}
\sum_{P\in\Q_0,P\subsetneq Q}
\int_{\bar\lambda_P}^{f_P}\lm{P\cap\{f>\lambda\}}\intd\lambda
\]
\end{lem}
Note that \(f_P>\bar\lambda_P\) implies \(P\in\Q_0\).

\begin{proof}[Proof of \cref{pro_densitylow}]
By \cref{lem_isoperimetric,cor_coarearange}
we have for each \(P\in\Q_0\) and \(P\subsetneq Q\) that
\begin{align*}
\int_{\bar\lambda_P}^{f_P}\lm{\{f>\lambda\}\cap P}\intd\lambda
&\leq
\sle(P)
\int_{\bar\lambda_P}^{f_P}\lm{\{f>\lambda\}\cap P}^{1-\f1d}\intd\lambda
\\
&\lesssim
\sle(P)
\int_{\bar\lambda_P}^{f_P}\sm{\mb{\{f>\lambda\}}\cap P}\intd\lambda
\\
&=
\sle(P)
\int_{x\in P:\bar\lambda_P<f(x)<f_P}|\nabla f|
\\
&=
\sle(P)
\int_Q|\nabla f|
\ind{P\times(\bar\lambda_P,f_P)}(x,f(x))\intd x
.
\end{align*}
We note that for any \(Q\in\Q\) we have \(\lambda_Q^\Q\geq\lambda_Q^\emptyset\)
and use \cref{cla_mostmasssparseabove}.
Then we apply the above calculation,
H\"older's inequality
and use that \((\bar\lambda_P,f_P)\) and \((\bar\lambda_Q,f_Q)\) are disjoint for \(P\subsetneq Q\),
\begin{align*}
&\sum_{Q\in\Q}
\Bigl(
\sle(Q)^{\f dp-1}(f_Q-\lambda_Q^\Q)
\Bigr)^p
\\
&\leq 2^{d+2}
\sum_{Q\in\Q}\Biggl(
\sle(Q)^{\f dp-1-d}
\sum_{P\in\Q_0,P\subsetneq Q}
\int_{\bar\lambda_P}^{f_P}\lm{\{f>\lambda\}\cap P}\intd\lambda
\Biggr)^p
\\
&\lesssim
\sum_{Q\in\Q}\Biggl(
\sle(Q)^{\f dp-1-d}
\int_Q|\nabla f|
\sum_{P\in\Q_0,P\subsetneq Q}\sle(P)
\ind{P\times(\bar\lambda_P,f_P)}(x,f(x))\intd x
\Biggr)^p
\\
&\leq
\sum_{Q\in\Q}\Biggl(
\sle(Q)^{\f dp-1-d+d(1-\f1p)}
\Biggl[
\int_Q|\nabla f|^p\biggl(
\sum_{P\in\Q_0,P\subsetneq Q}\sle(P)\ind{P\times(\bar\lambda_P,f_P)}(x,f(x))
\biggr)^p\intd x
\Biggr]^{\f1p}\Biggr)^p
\\
&=
\sum_{Q\in\Q}\Biggl(
\sle(Q)^{-1}\Biggl[
\sum_{P\in\Q_0,P\subsetneq Q}\sle(P)^p
\int_{(x,f(x))\in P\times(\bar\lambda_P,f_P)}|\nabla f|^p
\Biggr]^{\f1p}\Biggr)^p
\\
&=
\sum_{Q\in\Q}
\sle(Q)^{-p}
\sum_{P\in\Q_0,P\subsetneq Q}\sle(P)^p
\int_{(x,f(x))\in P\times(\bar\lambda_P,f_P)}|\nabla f|^p
\\
&=
\sum_{P\in\Q_0}
\sle(P)^p
\int_{x\in P:f(x)\in(\bar\lambda_P,f_P)}|\nabla f|^p
\sum_{Q\in\Q,Q\supsetneq P}\sle(Q)^{-p}
\\
&\leq\f1{2^p-1}
\sum_{P\in\Q_0}
\int_{x\in P:f(x)\in(\bar\lambda_P,f_P)}|\nabla f|^p\\
&\leq\f1{2^p-1}
\int|\nabla f|^p
.
\end{align*}
For \(p=1\) with \(\var f\) instead of \(\|\nabla f\|_1\)
we do not use \cref{cor_coarearange} or H\"older's inequality,
but interchange the order of summation first
and then apply \cref{lem_coarearange}.
\end{proof}

For a dyadic cube \(Q\) denote by \(\prt(Q)\) the dyadic parent cube of \(Q\).
\begin{lem}
\label{lem_disjointrepresentative}
Let \(1\leq p<d/\alpha\) and \(f\in L^p(\mathbb{R}^d)\) and let \(\varepsilon>0\).
Then there is a subset \(\tilde\Q_\alpha\) of \(\Q_\alpha\) such that
for each \(Q\in\Q_\alpha\) with \(\sle(Q)^\alpha f_Q>\varepsilon\) there is a \(P\in\tilde\Q_\alpha\) with
\(Q\subset\prt(P)\) and \(f_Q\leq 2^d f_P\).
Furthermore for any two \(Q,P\in\tilde\Q_\alpha\) one of the following holds.
\begin{enumerate}
\item
\label{it_parentsequal}
\(\prt(Q)=\prt(P)\).
\item
\label{it_parentsdisjoint}
\(\prt(Q)\) and \(\prt(P)\) don't intersect.
\item
\label{it_rangedisjoint}
\(f_Q/f_P\not\in(2^{-d},2^d)\).
\end{enumerate}
\end{lem}

\begin{proof}
Set \(\tilde\Q_\alpha^0\) to be the set of maximal cubes \(Q\) with \(\sle(Q)^\alpha f_Q>\varepsilon\).
For any dyadic cube \(Q\) with \(\sle(Q)^\alpha f_Q>\varepsilon\) we have
\[
\varepsilon
<
\sle(Q)^{\alpha-d}\int_Q f
\leq
\sle(Q)^{\alpha-d+d-\f dp}\Bigl(\int_Q f^p\Bigr)^{\f1p}
\leq
\sle(Q)^{\alpha-\f dp}\|f\|_p
\]
which implies
\begin{equation}
\label{eq_boundedsle}
\sle(Q)
<
(\|f\|_p/\varepsilon)^{(p^{-1}-\alpha/d)^{-1}}
.
\end{equation}
Hence
\[
\bigcup\tilde\Q_\alpha^0
=
\bigcup\{Q\in\Q_\alpha:\sle(Q)^\alpha f_Q>\varepsilon\}
.
\]
Assume we have already defined \(\tilde\Q_\alpha^n\).
Then define \(\tilde\Q_\alpha^{n+1}\) to be the set of maximal cubes \(Q\in\Q_\alpha\) with
\begin{equation}
\label{eq_fQmuchlarger}
f_Q
>
2^d\sup_{P\in\tilde\Q_\alpha^n:Q\subset\prt(P)}f_P
.
\end{equation}
Set \(\tilde\Q_\alpha=\tilde\Q_\alpha^0\cup\tilde\Q_\alpha^1\cup\ldots\).

Assume there is a cube \(Q\) with \(\sle(Q)^\alpha f_Q>\varepsilon\)
such that for all \(P\in\tilde\Q_\alpha\) with \(Q\subset\prt(P)\) we have \(f_Q>2^df_P\).
Then by \cref{eq_boundedsle} there is a maximal such cube \(Q\).
Furthermore there is a smallest \(P\in\tilde\Q_\alpha\) with \(Q\subset\prt(P)\)
and an \(n\) with \(P\in\tilde\Q_\alpha^n\).
But then \(Q\) is a maximal cube that satisfies \cref{eq_fQmuchlarger},
which implies \(Q\in\tilde\Q_\alpha^{n+1}\),
a contradiction.

If for \(Q,P\in\tilde\Q_\alpha\) neither \cref{it_parentsequal} nor \cref{it_parentsdisjoint} holds,
then after renaming we have \(\prt(Q)\subsetneq\prt(P)\).
Then \(P\) has been added to \(\tilde\Q_\alpha\) before \(Q\),
and since \(Q\subset\prt(P)\) this means \(f_Q>2^df_P\).
\end{proof}

\begin{lem}
\label{lem_disjointcubes}
Let \(1\leq p<\infty\) and \(f\in W^{1,p}(\mathbb{R}^d)\) and let \(\varepsilon>0\).
Let \(\Q\subset\Q_0\) be a set of dyadic cubes such that
\begin{enumerate}
\item
\label{it_largeraverage}
for each \(Q\in\Q\)
there is an ancestor cube \(p(Q)\supsetneq Q\) with \(\sle(p(Q))\leq\sle(Q)/\varepsilon\) and \(f_Q>2^\varepsilon f_{p(Q)}\),
\item
\label{it_disjoint}
and for any two distinct \(Q,P\in\Q\) such that \(p(Q)\) and \(p(P)\) intersect
we have \(f_Q/f_P\not\in(2^{-\varepsilon},2^\varepsilon)\).
\end{enumerate}
Then
\[
\Biggl(
\sum_{Q\in\Q}
(\sle(Q)^{\f dp-1}f_Q)^p
\Biggr)^{\f1p}
\lesssim_\varepsilon
\|\nabla f\|_p
.
\]
The endpoint \(p=\infty\) holds as well.
\end{lem}
\begin{proof}
We divide into two types of cubes and deal with them separately.
Denote
\begin{align*}
\Q_-
&=
\{Q\in\Q:
\lm{\{f>2^{-\varepsilon/3}f_Q\}\cap Q}<2^{-d-2}\lm Q
\}
,
\\
\Q_+
&=
\{Q\in\Q:
\lm{\{f>2^{-\varepsilon/3}f_Q\}\cap Q}\geq2^{-d-2}\lm Q
\}
.
\end{align*}
Let \(Q\in\Q_-\)
and recall \(\lambda_Q^{\Q}\) from \cref{pro_densitylow}.
Then since
\begin{align*}
\sup\{\lambda:\lm{\{f>\lambda\}\cap Q}<2^{-d-2}\lm Q\}
&\leq
2^{-\varepsilon/3}f_Q
,
\\
\sup\{f_P:P\in\Q,\ P\supsetneq Q\}
&\leq
2^{-\varepsilon}f_Q
\end{align*}
we have
\[
f_Q-\lambda_Q^{\Q}
\geq
(1-2^{-\varepsilon/3})f_Q
.
\]
Since \(\Q\subset\Q_0\) we conclude from \cref{pro_densitylow}
\[
\sum_{Q\in\Q_-}
\Bigl(
\sle(Q)^{\f dp-1}f_Q
\Bigr)^p
\leq
(1-2^{-\varepsilon/3})^{-p}
\sum_{Q\in\Q_-}
\Bigl(
\sle(Q)^{\f dp-1}
(f_Q-\lambda_Q^{\Q})
\Bigr)^p
\lesssim_{\varepsilon,p}
\|\nabla f\|_p^p
.
\]
%Note that for \(\Q_-\) we didn't need the reduction to \(\tilde\Q\), we could also have done it for the whole \(\Q\).

Let \(Q\in\Q_+\) and \(\lambda>2^{-2\varepsilon/3}f_Q\).
Since by \cref{it_largeraverage} we have \(2^{\varepsilon/3}f_{p(Q)}<2^{-2\varepsilon/3}f_Q\),
we obtain from Chebyshev's inequality
\begin{equation}
\label{eq_chebychev}
\lm{p(Q)\cap\{f>\lambda\}}
\\
\leq
2^{-\varepsilon/3}\lm{p(Q)}
.
\end{equation}
Since \(Q\in\Q_+\), for \(\lambda<2^{-\varepsilon/3}f_Q\) we have
\begin{equation}
\label{eq_muchmassinpQ}
2^{-d-2}\varepsilon^d\lm{p(Q)}
\leq
2^{-d-2}\lm Q
\leq
\lm{Q\cap\{f>\lambda\}}
\leq
\lm{p(Q)\cap\{f>\lambda\}}
.
\end{equation}
So for all
\(
2^{-2\varepsilon/3}f_Q
\leq
\lambda
\leq
2^{-\varepsilon/3}f_Q
\)
we can conclude by the isoperimetric inequality \cref{lem_isoperimetric}
and \cref{eq_chebychev,eq_muchmassinpQ}
that
\begin{align*}
\sm{\mb{\{f>\lambda\}}\cap p(Q)}^d
&\gtrsim
\min\{\lm{p(Q)\cap\{f>\lambda\}},\lm{p(Q)\setminus\{f>\lambda\}}\}^{d-1}
\\
&\geq
(\lm{p(Q)}\min\{\varepsilon^d2^{-d-2},1-2^{-\varepsilon/3}\})^{d-1}
\\
&\gtrsim_\varepsilon
\lm{p(Q)}^{d-1}
.
\end{align*}
Thus for each \(Q\in\Q_+\) by \cref{cor_coarearange} and H\"older's inequality we have
\begin{align*}
\int_{2^{-2\varepsilon/3}f_Q}^{2^{-\varepsilon/3}f_Q}
\sle(p(Q))^{d-1}
\intd\lambda
&\lesssim_\varepsilon
\int_{2^{-2\varepsilon/3}f_Q}^{2^{-\varepsilon/3}f_Q}
\sm{\mb{\{f>\lambda\}}\cap p(Q)}
\intd\lambda
\\
&=
\int_{x\in p(Q):f(x)\in(2^{-2\varepsilon/3},2^{-\varepsilon/3})f_Q}
|\nabla f|
\\
&\leq
\sle(p(Q))^{d-\f dp}
\Biggl(
\int_{x\in p(Q):f(x)\in(2^{-2\varepsilon/3},2^{-\varepsilon/3})f_Q}
|\nabla f|^p
\Biggr)^{\f1p}
.
\end{align*}
Now we use \cref{it_disjoint}
and conclude
\begin{align*}
\sum_{Q\in\Q_+}
\Bigl(
\sle(Q)^{\f dp-1}f_Q
\Bigl)^p
&\lesssim_{\varepsilon,p}
\sum_{Q\in\Q_+}
\Bigl(
\sle(p(Q))^{\f dp-1}f_{p(Q)}
\Bigr)^p
\\
&\lesssim_{\varepsilon,p}
\sum_{Q\in\Q_+}
\Biggl(
\sle(p(Q))^{\f dp-d}
\int_{2^{-2\varepsilon/3}f_Q}^{2^{-\varepsilon/3}f_Q}
\sle(p(Q))^{d-1}
\intd\lambda
\Biggr)^p
\\
&\lesssim_{\varepsilon,p}
\sum_{Q\in\Q_+}
\int_{x\in p(Q):f(x)\in(2^{-2\varepsilon/3},2^{-\varepsilon/3})f_Q}
|\nabla f|^p
\\
&\leq
\int|\nabla f|^p
.
\end{align*}
For \(p=1\) with \(\var f\) instead of \(\|\nabla f\|_1\)
we use \cref{lem_coarearange}
instead of
\cref{cor_coarearange} and H\"older's inequality.
For \(p=\infty\) let \(Q\in\Q\).
Then by the Sobolev-Poincar\'e inequality we have
\begin{align*}
\|\nabla f\|_\infty
\geq
\|\nabla f\|_{L^\infty(p(Q))}
&\gtrsim
\sle(p(Q))^{-d-1}
\int_{p(Q)}|f-f_{p(Q)}|
\\
&\geq
\sle(Q)^{-d-1}
\varepsilon^{d+1}
\int_Q|f-f_{p(Q)}|
\\
&\geq
\sle(Q)^{-d-1}
\varepsilon^{d+1}
\int_Qf-f_{p(Q)}
\\
&=
\sle(Q)^{-1}
\varepsilon^{d+1}
(f_Q-f_{p(Q)})
\\
&\geq
\sle(Q)^{-1}
\varepsilon^{d+1}(1-2^{-\varepsilon})
f_Q
.
\end{align*}

\end{proof}

\begin{proof}[Proof of \cref{eq_finitedyadiclinear}]
Let \(\varepsilon>0\) and 
\(\tilde\Q_\alpha\) be the set of cubes from \cref{lem_disjointrepresentative}.
Let \(Q\in\Q_\alpha\).
Then there is a \(P\in\tilde\Q_\alpha\)
with \(Q\subset\prt(P)\) and \(f_Q\leq 2^d f_P\).
Then \(f_Q\leq 4^d f_{\prt(P)}\).
Thus since \(\sle(Q)^\alpha f_Q>\sle(\prt(P))^\alpha f_{\prt(P)}\)
we have \(\sle(Q)>4^{-d/\alpha}\sle(\prt(P))\).
Thus for each \(P\) there are at most \(c_\alpha\) many \(Q\in\Q_\alpha\)
with \(Q\subset\prt(P)\) and \(f_Q\leq2^df_P\).
We conclude
\begin{align*}
\sum_{Q\in\Q_\alpha,\sle(Q)^\alpha f_Q>\varepsilon}
\Bigl(
\sle(Q)^{\f dp-1}f_Q
\Bigr)^p
&
\leq
\sum_{P\in\tilde\Q_\alpha}
\sum_{Q\in\Q_\alpha,\ Q\subset\prt(P),\ f_Q\leq2^df_P}
\Bigl(
\sle(Q)^{\f dp-1}f_Q
\Bigr)^p
\\
&
\lesssim_{\alpha,p} c_\alpha
\sum_{P\in\tilde\Q_\alpha}
\Bigl(
\sle(P)^{\f dp-1}f_P
\Bigr)^p
.
\end{align*}

For each dyadic cube \(P\in\{\prt(Q):Q\in\tilde\Q_\alpha\}\)
pick a \(Q\in\tilde\Q_\alpha\) with \(P=\prt(Q)\)
such that for all \(Q'\in\tilde\Q_\alpha\) with \(P=\prt(Q')\)
we have \(f_{Q'}\leq f_Q\).
Denote by \(\hat\Q_\alpha\) the set of all such dyadic cubes \(Q\).
Then
\begin{align*}
\sum_{Q\in\tilde\Q_\alpha}
\Bigl(
\sle(Q)^{\f dp-1}f_Q
\Bigr)^p
&\leq
\sum_{P\in\{\prt(Q):Q\in\tilde\Q_\alpha\}}
\sum_{Q\in\tilde\Q_\alpha:P=\prt(Q)}
\Bigl(
\sle(Q)^{\f dp-1}f_Q
\Bigr)^p
\\
&\leq
\sum_{P\in\{\prt(Q):Q\in\tilde\Q_\alpha\}}
2^d\sum_{Q\in\hat\Q_\alpha:P=\prt(Q)}
\Bigl(
\sle(Q)^{\f dp-1}f_Q
\Bigr)^p
\\
&=
2^d\sum_{Q\in\hat\Q_\alpha}
\Bigl(
\sle(Q)^{\f dp-1}f_Q
\Bigr)^p
\end{align*}
We want to show that \cref{lem_disjointcubes} applies to \(\hat\Q_\alpha\) with \(p(Q)=\prt(Q)\).
Since \(\hat\Q_\alpha\subset\Q_\alpha\) we have \(\hat\Q_\alpha\subset\Q_0\) by \cref{rem_largeralpha},
and \cref{it_largeraverage} follows from \(f_Q>2^\alpha f_{\prt(Q)}\).
For \cref{it_disjoint} let \(Q,P\in\hat\Q_\alpha\) be distinct such that \(\prt(Q)\) and \(\prt(P)\) intersect.
Since we have \(\prt(Q)\neq\prt(P)\), \cref{lem_disjointrepresentative} implies \(f_Q/f_P\not\in(2^{-d},2^d)\).
Thus by \cref{lem_disjointcubes} we have
\[
2^d\sum_{Q\in\hat\Q_\alpha}
\Bigl(
\sle(Q)^{\f dp-1}f_Q
\Bigr)^p
\lesssim_{\alpha,p}
\|\nabla f\|_p^p
.
\]

We have proven for every \(\varepsilon>0\) that
\[
\sum_{Q\in\Q_\alpha,\sle(Q)^\alpha f_Q>\varepsilon}
\Bigl(
\sle(Q)^{\f dp-1}f_Q
\Bigr)^p
\lesssim_{\alpha,p}
\|\nabla f\|_p^p
\]
with constant independent of \(\varepsilon\).
So we can let \(\varepsilon\) go to zero and conclude \cref{eq_finitedyadiclinear}.

For the endpoint \(p=\infty\) let \(Q\in\Q_\alpha\).
Then we use \(f_{\prt(Q)}\leq2^{-\alpha}f_Q\) and copy the proof of the endpoint in \cref{lem_disjointcubes}
with \(p(Q)=\prt(Q)\) and \(\varepsilon=1/2\).
\end{proof}

\section{Hardy-Littlewood Maximal Operator}
\label{sec_uncentered}
In this section we prove \cref{theo_mfdr}.

\subsection{Making the Balls Disjoint}

\begin{lem}
\label{lem_disjointballs}
Let 
\(\M_\alpha\in\{\Mc_\alpha,\Mu_\alpha\}\)
and
\(1\leq p<d/(1+\alpha+\beta)\)
and
\(f\in L^p(\mathbb{R}^d)\)
and let 
\(\varepsilon>0\).
Then for any \(c_1\geq2,c_2\geq1\) there is a set of balls \(\widetilde\B\subset\B_\alpha\) such that
for two balls \(B,A\in\widetilde\B\) we have
\(c_1B\cap c_1A=\emptyset\) or \(f_A/f_B\not\in(c_2^{-1},c_2)\),
and furthermore
\begin{align*}
&
\int_\varepsilon^\infty\lambda^{(p^{-1}-(1+\alpha+\beta)/d)^{-1}-1}
\lmb{
\bigcup
\bigl\{
B\in\B_\alpha:r(B)^{\alpha+\beta}f_B>\lambda
\bigr\}
}
\intd \lambda
\\
&\lesssim_{\alpha,\beta,p,c_1,c_2}
\biggl(
\sum_{B\in\widetilde\B}
\Bigl(
r(B)^{\f dp-1}f_B
\Bigr)^p
\biggr)^{(1-p(1+\alpha+\beta)/d)^{-1}}
.
\end{align*}
\end{lem}
\begin{proof}
Let \(B\in\B_\alpha\) with \(r(B)^{\alpha+\beta}f_B>\varepsilon\).
Then
\[
\varepsilon
<
r(B)^{\alpha+\beta} f_B
\leq
r(B)^{\alpha+\beta}\lm B^{-1}\lm B^{1-1/p}
\Bigl(\int_B f^p\Bigr)^{1/p}
\leq
\sigma_d^{-1/p}
r(B)^{\alpha+\beta-d/p}
\|f\|_p
,
\]
which means that \(r(B)\) is bounded by
\[
K
=
(\sigma_d^{-1/p}\|f\|_p/\varepsilon)^{1/(d/p-\alpha-\beta)}
.
\]
Define \(\B^0=\{B\in\B_\alpha:r(B)\in[1/2,1]K\}\).
Then for all \(B\in\B^0\)
we have that \(r(B)^\alpha f_B\) is uniformly bounded.
Inductively define a sequence of balls as follows.
For \(B_0,\ldots,B_{k-1}\) already defined
choose a ball \(B_k\in\B^0\) such that \(c_1B_k\) is disjoint from \(c_1B_0,\ldots,c_1B_{k-1}\)
and which attains at least half of
\[
\sup\{
f_B:
B\in\B^0,c_1B\cap(c_1B_0\cup\ldots\cup c_1B_{k-1})=\emptyset
\}
\]
if one exists.
Set \(\widetilde{\B^0}=\{B_0,B_1,\ldots\}\).
Then for all \(B\in\B^0\) we have that \(c_1B\) intersects \(\bigcup\{c_1B:B\in\widetilde{\B^0}\}\).
Define
\[
\cl{\B^0}=\{B(x,r)\in\B_\alpha:\exists A\in\widetilde{\B^0}\ A\subset B(x,5c_1r(A)),\ f_{B(x,r)}\leq c_2f_A\}
.
\]
Then \(\B^0\subset\cl{\B^0}\).
We proceed by induction.
For each \(n\in\mathbb{N}\) define
\[
\B^n
=
\bigl\{
B\in\B_\alpha\setminus(\cl{\B^0}\cup\ldots\cup\cl{\B^{n-1}}):r(B)\in[1/2,1]2^{-n}K
\bigr\}
,
\]
as above greedily select a sequence \(\widetilde{\B^n}\)
of balls \(B\in\B^n\) with almost maximal \(f_B\) such that for every already selected \(A\in\widetilde{\B^n}\)
we have \(c_1B\cap c_1A=\emptyset\),
and define
\[
\cl{\B^n}
=
\bigl\{
B(x,r)\in\B_\alpha:\exists A\in\widetilde{\B^n}\ A\subset B(x,5c_1r(A)),\ f_{B(x,r)}\leq c_2f_A
\bigr\}
.
\]
Note that we have \(\B^n\subset\cl{\B^n}\).
Finally set
\(\widetilde\B=\widetilde{\B^0}\cup\widetilde{\B^1}\cup\ldots\).
For \(A\in\widetilde\B\), we denote
\[
U_{A,\lambda}
=
\bigl\{
B(x,r)\in\B_\alpha: A\subset B(x,5c_1r(A)),\ f_{B(x,r)}\leq c_2f_A,r^{\alpha+\beta} f_{B(x,r)}>\lambda
\bigr\}
.
\]
Let \(\lambda>\varepsilon\) and \(B\in\B_\alpha\) with \(r(B)^{\alpha+\beta}f_B>\lambda\).
Then there is an \(n\) with \(B\in\cl{\B^n}\),
and hence a \(A\in\widetilde{\B^n}\) with \(B\in U_{A,\lambda}\).
Let \(A\in\widetilde\B\) and \(B(x,r)\in U_{A,\lambda}\).
Then \(A\subset B(x,5c_1r(A))\).
Since \(r\in R_\alpha f(x)\) we have
\[
r^\alpha f_{B(x,r)}
\geq
(5c_1r(A))^\alpha f_{B(x,5c_1r(A))}
\geq
(5c_1r(A))^\alpha(5c_1)^{-d} f_A
\]
which implies
\[
r
\geq
(5c_1)^{1-d/\alpha}r(A)(f_A/f_{B(x,r)})^{1/\alpha}
\geq
(5c_1)^{1-d/\alpha}c_2^{1/\alpha}r(A)
.
\]
Since \(r\leq5c_1r(A)\)
it follows that
\[
r^\beta
\leq
r(A)^\beta
\begin{cases}
(5c_1)^\beta,
&
\beta\geq0
,
\\
(5c_1)^{\beta-d\beta/\alpha}c_2^{\beta/\alpha},
&
\beta<0
.
\end{cases}
\]
Together with
\[
r^\alpha f_{B(x,r)}
\leq
(5c_1r(A))^\alpha
c_2f_A
\]
we obtain
\[
r^{\alpha+\beta} f_{B(x,r)}
\leq c_3
r(A)^{\alpha+\beta}f_A
,
\]
where
\[
c_3
=
\begin{cases}
(5c_1)^{\alpha+\beta}
c_2,
&
\beta\geq0
,
\\
(5c_1)^{\alpha+\beta-d\beta/\alpha}c_2^{1+\beta/\alpha},
&
\beta<0.
\end{cases}
\]
Thus \(U_{A,\lambda}\) is only nonempty if
\[
\lambda
< c_3
r(A)^{\alpha+\beta}f_A
.
\]
We can conclude
\begin{align*}
&
\int_\varepsilon^\infty
\lambda^{(p^{-1}-(1+\alpha+\beta)/d)^{-1}-1}
\lmb{
\bigcup\{B\in\B_\alpha:r(B)^{\alpha+\beta}f_B>\lambda\}
}
\intd\lambda
\\
&=
\int_\varepsilon^\infty
\lambda^{(p^{-1}-(1+\alpha+\beta)/d)^{-1}-1}
\lmb{
\bigcup_{A\in\widetilde\B}
\bigcup U_{A,\lambda}
}
\intd\lambda
\\
&\leq
\sum_{A\in\widetilde\B}
\int_\varepsilon^\infty
\lambda^{(p^{-1}-(1+\alpha+\beta)/d)^{-1}-1}
\lmb{
\bigcup U_{A,\lambda}
}
\intd\lambda
\\
&=
\sum_{A\in\widetilde\B}
\int_\varepsilon^{c_3 r(A)^{\alpha+\beta}f_A}
\lambda^{(p^{-1}-(1+\alpha+\beta)/d)^{-1}-1}
\lmb{
\bigcup U_{A,\lambda}
}
\intd\lambda
\\
&\leq
\sum_{A\in\widetilde\B}
(5c_1)^d\lm A
\int_\varepsilon^{c_3 r(A)^{\alpha+\beta}f_A}
\lambda^{(p^{-1}-(1+\alpha+\beta)/d)^{-1}-1}
\intd\lambda
\\
&\leq
(1/p-(1+\alpha+\beta)/d)
\sum_{A\in\widetilde\B}
(5c_1)^d\lm A
\Bigl(
c_3 r(A)^{\alpha+\beta}f_A
\Bigr)^{(p^{-1}-(1+\alpha+\beta)/d)^{-1}}
\\
&=
(1/p-(1+\alpha+\beta)/d)
(5c_1)^d
c_3^{(p^{-1}-(1+\alpha+\beta)/d)^{-1}}
\sigma_d
\sum_{A\in\widetilde\B}
\Bigl(
r(A)^{\f dp-1}f_A
\Bigr)^{(p^{-1}-(1+\alpha+\beta)/d)^{-1}}
\\
&\leq
(1/p-(1+\alpha+\beta)/d)
(5c_1)^d
c_3^{(p^{-1}-(1+\alpha+\beta)/d)^{-1}}
\sigma_d
\biggl(
\sum_{A\in\widetilde\B}
\Bigl(
r(A)^{\f dp-1}f_A
\Bigr)^p
\biggr)^{(1-p(1+\alpha+\beta)/d)^{-1}}
.
\end{align*}
\end{proof}

\subsection{Transfer to Dyadic Cubes}

In this subsection we pass from disjoint balls to dyadic cubes
and then conclude \cref{theo_mfdr}
using a result from the dyadic setting.

\begin{rem}
\label{lem_dyadicgrids}
There are \(3^d\) dyadic grids \(\D_1,\ldots,\D_{3^d}\) such that each ball \(B\) is contained in a dyadic cube \(Q_B\in\D=\D_1\cup\ldots\cup\D_{3^d}\) with \(\sle(Q)\lesssim r(B)\).
\end{rem}

\begin{lem}
\label{lem_ballsimcube}
Let \(\M_\alpha\in\{\Mc_\alpha,\Mu_\alpha\}\)
and \(f\in L^1_\loc(\mathbb{R}^d)\).
Then for each \(B\in\B_\alpha\)
we have
\(f_{Q_B}\sim f_B\)
and
\(\sle(Q_B)\sim r(B)\).
\end{lem}
\begin{proof}
Let \(x\) be the center of \(B\),
and \(Q_B\) be the cube from \cref{lem_dyadicgrids},
and \(A=B(x,\sqrt d\sle(Q))\).
Then \(r(B)\sim\sle(Q_B)\sim r(A)\)
and
\(f_B\lesssim f_{Q_B}\lesssim f_A\).
Since \(B\in\B_\alpha\) we also have
\(r(A)^\alpha f_A<r(B)^\alpha f_B\)
and therefore conclude \(f_{Q_B}\lesssim f_A\lesssim f_B\).
\end{proof}

\begin{lem}
\label{lem_lowerancestor}
Let \(\M_\alpha\in\{\Mc_\alpha,\Mu_\alpha\}\)
and \(f\in L^1_\loc(\mathbb{R}^d)\).
For each \(\alpha>0\) and \(B\in\B_\alpha\) and cube \(P\supset Q_B\) we have
\(
\sle(P)^\alpha
f_P
\lesssim_\alpha
\sle(Q_B))^\alpha
f_{Q_B}
.
\)
\end{lem}
\begin{proof}
For \(x\) the center of \(B\) define \(A=B(x,\sqrt d\sle(P))\).
Then from \(f_P\lesssim f_A\)
and \(r(A)^\alpha f_A<r(B)^\alpha f_B\)
and \(f_B\lesssim f_{Q_B}\)
we obtain
\(
\sle(P)^\alpha f_P
\lesssim
s^\alpha f_{B(x,s)}
<
r^\alpha f_{B(x,r)}
\lesssim_\alpha
\sle(Q_{B(x,r)})^\alpha f_{Q_{B(x,r)}}
.
\)
\end{proof}

\begin{proof}[Proof of \cref{theo_mfdr}]
For \(B\in\B_\alpha\) denote by \(P_B\) the largest cube that attains \(\max_{P\supset Q_B}f_P\).
Then \(P_B\in\Q_0\) and
by \cref{lem_ballsimcube,lem_lowerancestor}
we have \(\sle(P_B)\sim_\alpha r(B)\)
and \(f_{P_B}\sim_\alpha f_B\).
By \cref{lem_lowerancestor} there further exists a cube \(p(P_B)\supset P_B\) with \(f_{p(P_B)}\leq f_{P_B}/2\)
and \(\sle(p(P_B))\lesssim_\alpha\sle(P_B)\).

Let \(\varepsilon>0\) and let \(\widetilde\B\) be the set of balls from \cref{lem_disjointballs}.
By \cref{lem_ballsimcube,lem_lowerancestor} there are \(c_1,c_2\) such that for any two distinct \(B,A\in\widetilde\B\)
we have that \(p(P_B)\) and \(p(P_A)\) are disjoint
or \(f_{P_B}/f_{P_A}\not\in(1/2,2)\).
Define \(\Q=\{P_B:B\in\widetilde\B\}\).
By the layer cake formula
and \cref{lem_disjointballs,lem_ballsimcube} we have
\begin{align*}
&
\int(\M_{\alpha,\beta}f)^{(p^{-1}-(1+\alpha+\beta)/d)^{-1}}
\\
&=
{(p^{-1}-(1+\alpha+\beta)/d)^{-1}}
\int_0^\infty
\lambda^{(p^{-1}-(1+\alpha+\beta)/d)^{-1}-1}
\lm{\{\M_{\alpha,\beta} f>\lambda\}}
\intd\lambda
\\
&=
{(p^{-1}-(1+\alpha+\beta)/d)^{-1}}
\lim_{\varepsilon\rightarrow0}
\int_\varepsilon^\infty
\lambda^{(p^{-1}-(1+\alpha+\beta)/d)^{-1}-1}
\lmb{
\bigcup\{B\in\B_\alpha:r(B)^{\alpha+\beta}f_B>\lambda\}
}
\intd\lambda
\\
&
\lesssim_{\alpha,\beta,p}
\lim_{\varepsilon\rightarrow0}
\biggl(
\sum_{B\in\widetilde\B}
\Bigl(
r(B)^{\f dp-1}f_B
\Bigr)^p
\biggr)^{(1-p(1+\alpha+\beta)/d)^{-1}}
\\
&\sim_{\alpha,\beta,p}
\lim_{\varepsilon\rightarrow0}
\biggl(
\sum_{Q\in\Q}
\Bigl(
\sle(Q)^{\f dp-1}f_Q
\Bigr)^p
\biggr)^{(1-p(1+\alpha+\beta)/d)^{-1}}
.
\end{align*}

For each \(i=1,\ldots,3^d\) we apply \cref{lem_disjointcubes} to \(\Q\cap\D_i\) and obtain
\[
\sum_{Q\in\Q}
\Bigl(
\sle(Q)^{\f dp-1}f_Q
\Bigr)^p
=
\sum_{i=1}^{3^d}\sum_{Q\in\Q\cap\D_i}
\Bigl(
\sle(Q)^{\f dp-1}f_Q
\Bigr)^p
\lesssim_{\alpha,\beta,p}
\|\nabla f\|_p^p
.
\]

For the endpoint \(p=d/(1+\alpha+\beta)\) we use \(\|\M_{\alpha,\beta} f\|_\infty=\sup_{B\in\B_\alpha}r(B)^{\alpha+\beta}f_B\).
Let \(B\in\B_\alpha\).
Then \(f_{2B}\leq2^{-\alpha}f_B\) and we have by the Sobolev-Poincar\'e inequality
\begin{align*}
\|\nabla f\|_{d/(1+\alpha+\beta)}
\geq
\biggl(
\int_{2B}|\nabla f|^{d/(1+\alpha+\beta)}
\biggr)^{(1+\alpha+\beta)/d}
&\gtrsim
r(2B)^{\alpha+\beta-d}
\int_{2B}|f-f_{2B}|
\\
&\geq
2^{\alpha+\beta-d}
r(B)^{\alpha+\beta-d}
\int_B|f-f_{2B}|
\\
&\geq
2^{\alpha+\beta-d}
r(B)^{\alpha+\beta-d}
\int_B(f-f_{2B})
\\
&=
\sigma_d
2^{\alpha+\beta-d}
r(B)^{\alpha+\beta}
(f_B-f_{2B})
\\
&\geq
\sigma_d
2^{\alpha+\beta-d}
r(B)^{\alpha+\beta}
(1-2^{-\alpha})
f_B
.
\end{align*}
\end{proof}

\bibliographystyle{plain}
\bibliography{bib}

\end{document}